\documentclass[11pt]{amsart}

\usepackage{amsmath}
\usepackage{amssymb}
\usepackage{amsfonts}
\usepackage{amsthm}
\usepackage{enumerate}
\usepackage{hyperref}
\usepackage{color}
\usepackage{amscd}
\usepackage{verbatim}
\usepackage{calc}

\headsep=1cm \numberwithin{equation}{section}
\theoremstyle{plain}
\newtheorem{thm}{Theorem}[section]

\newtheorem{prop}[thm]{Proposition}
\newtheorem{lem}[thm]{Lemma}

\theoremstyle{definition}
\newtheorem{defn}[thm]{Definition}

\newtheorem{exmp}[thm]{Example}
\newtheorem{rem}[thm]{Remark}
\newtheorem{f}[thm]{ }

\begin{document}
\bibliographystyle{amsplain}

\title[Complexes of $C$-projective modules]{Complexes of $C$-projective modules}

\bibliographystyle{amsplain}

\author[E. Amanzadeh]{Ensiyeh Amanzadeh $^*$}
\address[Ensiyeh Amanzadeh]{Faculty of Mathematical Sciences and Computer,
Kharazmi University,  Tehran,  Iran. }
  \email{en.amanzadeh@gmail.com}

\author[M. T. Dibaei]{Mohammad T. Dibaei}
\address[Mohammad Taghi Dibaei]{Faculty of Mathematical Sciences and Computer,
Kharazmi University,  Tehran,  Iran; and School of Mathematics,
Institute for Research in Fundamental Sciences (IPM),  P.O. Box:
19395-5746,  Tehran,  Iran.}
  \email{dibaeimt@ipm.ir}

\thanks{$^*$Corresponding author}


\begin{abstract}
 Inspired by a recent work of Buchweitz and Flenner, we show that, for a semidualizing bimodule $C$,  $C$--perfect complexes
 have the ability to detect when a ring is strongly regular.
It is shown that there exists a class of modules which admit minimal resolutions of $C$--projective modules.
\end{abstract}


\maketitle


\section{Introduction} \label{sec1}
Let $R$ be a left and right noetherian ring (not necessarily
commutative), all modules left $R$--modules and $C$ a semidualizing
$(R, R)$--bimodule (Definition~\ref{d1}). A complex $X_\bullet$ of
$R$--modules is said to be $C$--{\it perfect} if it is
quasiisomorphic to a finite complex
$$T_\bullet= 0\longrightarrow C\otimes_R P_n\longrightarrow C\otimes_R P_{n-1}\longrightarrow\cdots\longrightarrow
C\otimes_R P_1\longrightarrow C\otimes_R P_0\longrightarrow 0,$$
where each $P_i$ is a finite (i.e. finitely generated) projective
$R$--module. The {\it width} of such a $C$--perfect complex
$X_\bullet$, denoted by $\mathrm{wd}(X_\bullet)$, is defined to be
the minimal length $n$ of a complex $T_\bullet$ satisfying the above
conditions. Recall from \cite{bf}, a ring $R$ is called {\it
strongly regular} whenever there exists a non-negative integer $r$
such that every $R$--perfect complex is quasiisomorphic to a direct
sum of $R$--perfect complexes of width $\leqslant r$. Buchweitz and
Flenner, in \cite{bf}, characterize the commutative noetherian rings
which are strongly regular.

Our first objective is to detect when a ring is strongly
regular by means of $C$--perfect complexes (Theorem~\ref{J}).
We also prove that $C$--projective modules (i.e. modules of the form
$C\otimes_R P$ with $P$ projective) have the ability to detect when
a ring is hereditary (Proposition \ref{H}).

Our second goal is to find a class of $R$--modules which admit minimal resolutions of
$C$--projective modules (see Theorem \ref{E}).

\section{Preliminaries} \label{sec2}

Throughout, $R$ is a left and right noetherian ring (not necessarily
commutative) and let all $R$--modules be left $R$--modules. Right
$R$--modules are identified with left modules over the opposite ring
$R^\mathrm{op}$. An $(R, R)$--{\it bimodule} $M$ is both left and
right $R$--module with compatible structures.

\begin{defn}\label{d1} \cite[Definition 2.1]{hw}
An $(R, R)$--bimodule $C$ is {\it semidualizing} if it is a finite $R$--module, finite $R^{\mathrm{op}}$--module,
and the following conditions hold.

(1) The homothety map $R\stackrel{^R\gamma}{\longrightarrow}\mathrm{ Hom}_{R^\mathrm{ op}}(C, C)$ is an isomorphism.

(2) The homothety map $R\stackrel{\gamma^R}{\longrightarrow}\mathrm{ Hom}_R(C, C)$ is an isomorphism.

(3) $\mathrm{ Ext}^{\geqslant 1}_R(C, C)=0$.

(4) $\mathrm{ Ext}^{\geqslant 1}_{R^\mathrm{ op}}(C, C)=0$.
\end{defn}
Assume that $R$ is a commutative noetherian ring, then the above definition agrees with the definition of
semidualizing $R$--module (see e.g. \cite[2.1]{hw}).
Also, every finite projective $R$--module of rank 1 is semidualizing (see \cite[Corollary 2.2.5]{s-w}).

\begin{defn} \cite[Definition 3.1]{hw}
A semidualizing $(R, R)$--bimodule $C$ is said to be {\it faithfully semidualizing}
if it satisfies the following conditions

(a) If $\mathrm{Hom}_R(C, M)=0$, then $M=0$ for any $R$--module $M$;

(b) If $\mathrm{Hom}_{R^\mathrm{ op}}(C, N)=0$, then $N=0$ for any $R^\mathrm{ op}$--module $N$.
\end{defn}

Note that over a commutative noetherian ring, all semidualizing modules are faithfully semidualizing,
by \cite[Proposition 3.1]{hw}.

For the remainder of this section $C$ denotes a semidualizing $(R, R)$--bimodule.
The following class of modules is already appeared in, for example, \cite{hj}, \cite{hw}, and \cite{tw}.
\begin{defn}
An $R$--module is called $C$--{\it projective} if it has the form $C\otimes_R P$
for some projective $R$--module $P$. The class of (resp. finite) $C$--projective modules is denoted
by ${\mathcal P}_C$ (resp. ${\mathcal P}^f_C$).
\end{defn}

\begin{f}
A complex $A$ of $R$--modules is called $\mathrm{ Hom}_R(\mathcal{P}_C, -)$--exact
if $\mathrm{ Hom}_R(C\otimes_R P, A)$ is exact for each projective $R$--module $P$.
The term $\mathrm{ Hom}_R(-, \mathcal{P}_C)$--exact is defined dually.
\end{f}

For the notations in the next fact one may see \cite[Definitions 1.4 and 1.5]{wsw1}
\begin{f}\label{d2}
A ${\mathcal P}_C$--{\it resolution} of an $R$--module $M$ is a complex $X$ in ${\mathcal P}_C$
with $X_{-n}=0=\mathrm{H}_n(X)$ for all $n>0$ and $M\cong \mathrm{H}_0(X)$.
The following exact sequence is the {\it augmented} ${\mathcal P}_C$--{\it resolution} of $M$
associated to $X$:
$$X^+=\cdots \stackrel{\partial^X_2}{\longrightarrow} C\otimes_R P_1 \stackrel{\partial^X_1}
{\longrightarrow}C\otimes_R P_0 \longrightarrow M\longrightarrow 0.$$
A ${\mathcal P}_C$--resolution $X$ of $M$ is called $proper$ if in addition $X^+$ is
$\mathrm{Hom}_R({\mathcal P}_C, -)$--exact.

The ${\mathcal P}_C$--{\it projective dimension} of $M$ is the quantity

$\hspace{1.2cm}$ ${\mathcal P}_C$--$\mathrm{pd}(M)=\inf \{\sup \{ n\geqslant0\ |\ X_n\neq0\}\ |\ X$
is an ${\mathcal P}_C$--resolution of $M\}$.

The objects of ${\mathcal P}_C$--projective dimension 0 are exactly $C$--projective $R$--modules.

The notion ({\it proper}) ${\mathcal P}_C$--{\it coresolution} is defined dually.
The {\it augmented} ${\mathcal P}_C$--{\it coresolution} associated to a ${\mathcal P}_C$--coresolution $Y$ is denoted by $^+Y$.
\end{f}

In \cite{tw}, the authors proved the following proposition for a commutative ring $R$. However, by an easy inspection, one can see that it is true even if $R$ is non-commutative.
\begin{prop}\label{p1}
Assume that $C$ is a faithfully semidualizing $(R, R)$--bimodule and that $M$ is an $R$--module.
The following statements hold true.
\begin{itemize}
\item[(a)]\cite[Corollary 2.10(a)]{tw}  The inequality ${\mathcal P}_C$--$\mathrm{pd}(M)\leqslant n$ holds if and only if there is a complex
$$ 0\longrightarrow C\otimes_R P_n\longrightarrow \cdots \longrightarrow C\otimes_R P_0 \longrightarrow M\longrightarrow 0$$
which is $\mathrm{ Hom}_R({\mathcal P}_C,  -)$--exact.
\item[(b)]\cite[Theorem 2.11(a)]{tw}  $\mathrm{ pd}_R(M)=\mathcal{P}_C$--$\mathrm{ pd}_R(C\otimes_R M)$.
\item[(c)] \cite[Theorem 2.11(c)]{tw} $\mathcal{P}_C$--$\mathrm{ pd}_R(M)=\mathrm{ pd}_R(\mathrm{ Hom}_R(C, M))$.
\end{itemize}
\end{prop}

\begin{rem}\label{f1}
By \cite[Proposition 5.3]{hw} the class ${\mathcal P}_C$ is
precovering, that is,  for an $R$--module $M$, there exists a
projective $R$--module $P$ and a homomorphism $\phi:
C\otimes_RP\rightarrow M$ such that, for every projective $Q$, the
induced map
$$\mathrm{ Hom}_R(C\otimes_RQ, C\otimes_RP)\stackrel{\mathrm{ Hom}_R(C\otimes_RQ, \phi)}{-\hspace{-0.2cm}-\hspace{-0.2cm}-\hspace{-0.2cm}-\hspace{-0.2cm}-\hspace{-0.2cm}-\hspace{-0.2cm}\longrightarrow}
\mathrm{ Hom}_R(C\otimes_RQ, M)$$
is surjective. Then one can iteratively take precovers to construct a complex
\begin{equation}\tag{\ref{f1}.1}{W = \cdots\stackrel{{\partial}_2^{X}}{\longrightarrow }
C\otimes_R P_1\stackrel{{\partial}_1^{X}}{\longrightarrow }C\otimes_R P_0\longrightarrow 0}
\end{equation}
such that $W^+$ is $\mathrm{ Hom}_R({\mathcal P}_C,  -)$--exact, where
$$W^+ = \cdots\stackrel{{\partial}_2^{X}}{\longrightarrow }C\otimes_R P_1\stackrel{{\partial}_1^{X}}{\longrightarrow }
C\otimes_R P_0 \stackrel{\phi}{\longrightarrow } M\longrightarrow 0.$$
For the notions precovering, covering, preenveloping and enveloping one can see \cite{ej}.
\end{rem}

Note that if $C$ is faithfully semidualizing $(R, R)$--bimodule and $M$ is an $R$--module,
then, by Proposition \ref{p1}(a), ${\mathcal P}_C$--$\mathrm{pd}(M)$ is equal to the length of
the shortest complex as (\ref{f1}.1).
Thus for any $R$--module $M$, the quantity ${\mathcal P}_C$--projective dimension of $M$,
defined in \cite{hw} and \cite{tw}, is equal to ${\mathcal P}_C$--$\mathrm{pd}(M)$ in \ref{d2}.


\section{Results} \label{sec3}

A ring $R$ is (left) hereditary if every left ideal is projective. The Cartan-Eilenberg theorem \cite[Theorem 4.19]{r}
shows that $R$ is hereditary if and only if every submodule of a projective module is projective.
We show that the quality of being hereditary can be detected by $C$--projective modules, which is interesting on its own.

\begin{prop}\label{H}
Assume that $C$ runs trough the class of faithfully semidualizing $(R, R)$--bimodules.
The following statements are equivalent.
\begin{itemize}
\item[(i)] $R$ is left hereditary.
\item[(ii)] For any $C$, every submodule of a $C$--projective $R$--module is also $C$--projective.
\item[(iii)] There exists a $C$ such that every submodule of a $C$--projective $R$--module is also $C$--projective.
\end{itemize}
\end{prop}

\begin{proof}
(i)$\Rightarrow$(ii). Let $C$ be a faithfully semidualizing bimodule and $N$ a submodule of $C\otimes_R P$, where $P$ is a projective $R$--module.
Then one gets the exact sequence $0 \longrightarrow \mathrm{ Hom}_R(C, N) \longrightarrow P$.
As $R$ is left hereditary, $\mathrm{ Hom}_R(C, N)$ is a projective $R$--module. By Proposition~\ref{p1}(c), $\mathcal{P}_C$--$\mathrm{pd}(N)=\mathrm{pd}(\mathrm{ Hom}_R(C, N))=0$.

(ii)$\Rightarrow$(iii) is immediate.

(iii)$\Rightarrow$(i). As every submodule of a $C$--projective $R$--module is $C$--projective, for any $R$--module $M$
one has $\mathcal{P}_C$--$\mathrm{pd}(M)\leqslant 1$.
Then for any $R$--module $N$ one gets $\mathrm{pd}(N)=\mathcal{P}_C$--$\mathrm{pd}(C\otimes_R N)\leqslant 1$,
by Proposition~\ref{p1}(b). It follows that every submodule of a projective is projective and so,
by \cite[Theorem 4.19]{r}, $R$ is left hereditary.
\end{proof}

\begin{defn}\label{d5}
A complex $X_\bullet$ of $R$--modules is called $C$--{\it perfect} if it is quasiisomorphic to a finite complex
$$T_\bullet= 0 \longrightarrow C\otimes_R P_n\longrightarrow C\otimes_R P_{n-1}\longrightarrow\cdots
\longrightarrow C\otimes_R P_1\longrightarrow C\otimes_R
P_0\longrightarrow 0,$$ where $P_i$ are finite projective
$R$--modules. The {\it width} of such a $C$--perfect complex
$X_\bullet$, denoted by $\mathrm{wd}(X_\bullet)$, is defined to be
the minimal length $n$ of a complex $T_\bullet$ satisfying the above
conditions. A $C$--perfect complex $X_\bullet$ is called {\it indecomposable} if
it is not quasiisomorphic to a direct sum of two non-trivial
$C$--perfect complexes.
\end{defn}

\begin{defn}\label{d6}\cite[Definition 1.1]{bf}
A ring $R$ is called {\it strongly} $r$--{\it regular} if every
perfect complex over $R$ is quasiisomorphic to a direct sum of
perfect complexes of width $\leqslant r$. If $R$ is strongly
$r$--regular for some $r$ then it will be called {\it strongly
regular}.
\end{defn}
\begin{rem}
As Professor Ragnar-Olaf Buchweitz kindly pointed out in his
personal communication with the authors, in \cite{bf} it should be
added the blanket statement that rings are noetherian and modules
are finite. Thus Definition~\ref{d6} agrees with \cite[Definition 1.1]{bf}.
Indeed, over a noetherian ring every perfect complex has bounded and finite homology.
\end{rem}

Note that a hereditary ring $R$ is strongly 1-regular, see \cite[Remark 1.2]{bf}.

In order to bring the results Theorem \ref{J} and Proposition
\ref{p2}, we quote some preliminaries.

\begin{defn}\cite[III.3.2(b)]{gm} and \cite[Definition 2.2.8]{cf}
Let $\alpha: A\rightarrow B$ be a morphism of $R$--complexes. The {\it mapping cone} of $\alpha$,
$\mathrm{Cone}(\alpha)$, is a complex which is given by
$$(\mathrm{Cone}(\alpha))_n=B_n\oplus A_{n-1} \ \ \ \mathrm{and}\ \ \ \ \ \partial_n^{\mathrm{Cone}(\alpha)}=
\left(\begin{array}{cccc}\partial_n^B & \alpha_{n-1}\\ 0 &
-\partial_{n-1}^A\\ \end{array}\right).$$
\end{defn}

It easy to see that the following lemma is also true if $R$ is
non-commutative.
\begin{lem}\label{L10}
Let $\alpha: A\rightarrow B$ be a morphism of $R$--complexes and $M$ be an $R$--module.
The following statements hold true.
\begin{itemize}
\item[(a)] \cite[Lemma 2.2.10]{cf} The morphism $\alpha$ is a quasiisomorphism if and only if $\mathrm{Cone}(\alpha)$ is acyclic.
\item[(b)] \cite[Lemma 2.3.11]{cf} $\mathrm{Cone}(\mathrm{Hom}_R(M, \alpha))\cong \mathrm{Hom}_R(M, \mathrm{Cone}(\alpha))$.
\item[(c)] \cite[Lemma 2.4.11]{cf} $\mathrm{Cone}(M\otimes_R \alpha)\cong M\otimes_R \mathrm{Cone}(\alpha)$.
\end{itemize}
\end{lem}

\begin{rem}\label{R10}
Let $C$ be a semidualizing $(R, R)$--bimodule. Assume that\\
$X=0\rightarrow X_n \rightarrow X_{n-1} \rightarrow \cdots
\rightarrow X_1 \rightarrow X_0 \rightarrow 0$ is an exact complex
of $R$--modules.
\begin{itemize}
\item[(a)] If each $X_i$ is a projective $R$--module, then it is easy to see that the induced complex $C\otimes_R X$ is exact.
\item[(b)] If each $X_i$ is a $C$--projective $R$--module, then the induced complex $\mathrm{Hom}_R(C, X)$ is exact,
since $\mathrm{Ext}_R^{\geqslant 1} (C, X_i)=0$.
\end{itemize}
\end{rem}
\begin{thm}\label{J}
The following statements are equivalent.
\begin{itemize}
\item[(i)] $R$ is strongly $r$--regular.
\item[(ii)] For any faithfully semidualizing bimodule $C$, every $C$--perfect complex is quasiisomorphic
to a direct sum of $C$--perfect complexes of width $\leqslant r$.
\item[(iii)] There exists a faithfully semidualizing bimodule $C$ such that every $C$--perfect complex
is quasiisomorphic to a direct sum of $C$--perfect complexes of width $\leqslant r$.
\end{itemize}
\end{thm}
\begin{proof}
(i)$\Rightarrow$(ii). Let $R$ be strongly $r$--regular, $C$ a
faithfully semidualizing bimodule. Assume that $X_\bullet$ is a
$C$--perfect complex. Then, by Definition~\ref{d5}, there exists a
finite complex $$T_\bullet= 0 \longrightarrow C\otimes_R
P_n\longrightarrow C\otimes_R P_{n-1}\longrightarrow\cdots
\longrightarrow C\otimes_R P_0\longrightarrow 0,$$ such that each
$P_i$ is a finite projective $R$--module and $X_\bullet$ is
quasiisomorphic to $T_\bullet$. Therefore $\mathrm{Hom}_R(C,
T_\bullet)\cong 0 \longrightarrow P_n\longrightarrow
P_{n-1}\longrightarrow\cdots \longrightarrow P_0\longrightarrow 0$
is a perfect complex. By Definition~\ref{d6}, there is a
quasiisomorphism $\alpha:\mathrm{Hom}_R(C,
T_\bullet)\stackrel{\simeq}{\longrightarrow }\bigoplus_{i=1}^s
F_\bullet^{(i)}$, where each $F_\bullet^{(i)}$ is a perfect complex
of width $\leqslant r$. We may assume that each $F_\bullet^{(i)}$ is
a finite complex of finite projective $R$--modules. By
Lemma~\ref{L10}(a), $\mathrm{ Cone}(\alpha)$ is acyclic. As
$\mathrm{Cone}(\alpha)$ is a finite complex of projective
$R$--modules, Remark~\ref{R10} implies that the complex $C\otimes_R
\mathrm{ Cone}(\alpha)$ is acyclic. By Lemma~\ref{L10}, the complex
$\mathrm{ Cone} (C\otimes_R\alpha)$ is acyclic too and so
$C\otimes_R\alpha$ is quasiisomorphism. Therefore $T_\bullet$ is
quasiisomorphic to $\bigoplus_{i=1}^s C\otimes_R F_\bullet^{(i)}$.
Note that each $C\otimes_R F_\bullet^{(i)}$ is a $C$--perfect
complex of width $\leqslant r$.

(ii)$\Rightarrow$(iii) is immediate.

(iii)$\Rightarrow$(i). Let $Y_\bullet$ be a perfect complex. Then,
by Definition~\ref{d5}, there is a finite complex $F_\bullet= 0
\longrightarrow  P_m\longrightarrow  P_{m-1}\longrightarrow\cdots
\longrightarrow P_0\longrightarrow 0$ of finite projective modules
which is quasiisomorphic to $Y_\bullet$. As $C\otimes_R F_\bullet$
is a $C$--perfect complex, our assumption implies that there is a
quasiisomorphism $\beta:C\otimes_R
F_\bullet\stackrel{\simeq}{\longrightarrow }\bigoplus_{i=1}^s
T^{(i)}_\bullet$, where each $T^{(i)}_\bullet$ is a $C$--perfect
complex of width $\leqslant r$. We may assume that, for each $i$,
$$T^{(i)}_\bullet = 0 \longrightarrow C\otimes_R
P_{n_i}^{(i)}\longrightarrow\cdots \longrightarrow C\otimes_R
P_0^{(i)}\longrightarrow 0$$ where each $P_j^{(i)}$ is a finite
projective $R$--module. Similar to the proof of
(i)$\Rightarrow$(ii), one observes that $\mathrm{ Hom}_R(C, \beta)$
is a quasiisomorphism. Therefore $F_\bullet$ is quasiisomorphic to
$\bigoplus_{i=1}^s \mathrm{ Hom}_R(C, T^{(i)}_\bullet)$. Note that
each $\mathrm{ Hom}_R(C, T^{(i)}_\bullet)$ is a perfect complex of
width $\leqslant r$. Thus $R$ is strongly $r$--regular.
\end{proof}

In \cite[Section 1]{am}, Avramov and Martsinkovsky define a general notion of minimality
for complexes: A complex $X$ is {\it minimal} if every homotopy equivalence
$\sigma : X \longrightarrow X$ is an isomorphism.
In \cite[Lemma 4.8]{w}, it is proved that, over a commutative local ring $R$ with maximal ideal $\mathfrak{m}$,
a complex $X$ consisting of modules in $\mathcal P_C^f$ is minimal if and only if $\partial^X(X)\subseteq {\frak m}X$.

In consistent to \cite[Lemma 1.6]{bf} we prove the following
proposition.
\begin{prop}\label{p2}
Let $R$ be a commutative noetherian local ring, $C$ a semidualizing $R$--module. The following
statements hold true.

{\rm (a)} Every $C$--perfect complex $X_\bullet$ is quasiisomorphic to a minimal finite complex
$$T_\bullet= 0 \longrightarrow C\otimes_R F_n\longrightarrow C\otimes_R F_{n-1}\longrightarrow\cdots
\longrightarrow C\otimes_R F_1\longrightarrow C\otimes_R F_0\longrightarrow 0,$$
where each $F_i$ is finite free $R$--module.

{\rm (b)} If two minimal finite complexes of modules of the form $C^m=\oplus^m C$ are quasiisomorphic,
then they are isomorphic.
\end{prop}
\begin{proof}
{\rm (a).} By Definition~\ref{d5}, a $C$--perfect complex $X_\bullet$ is quasiisomorphic to a finite complex
$$T_\bullet= 0 \longrightarrow C\otimes_R P_n\longrightarrow C\otimes_R P_{n-1}\longrightarrow\cdots
\longrightarrow C\otimes_R P_1\longrightarrow C\otimes_R
P_0\longrightarrow 0,$$ where each $P_i$ is a finite free $R$--module.
The complex $\mathrm{Hom}_R(C, T_\bullet)$ is a perfect complex and
so, by \cite[Lemma 1.6(1)]{bf}, there exist a minimal finite complex
$F_\bullet$ of finite free $R$--modules and a quasiisomorphism
$\alpha: \mathrm{Hom}_R(C,
T_\bullet)\stackrel{\simeq}{\longrightarrow } F_\bullet$.
As in the proof of Theorem~\ref{J}, it follows that
$C\otimes_R\alpha: C\otimes_R\mathrm{Hom}_R(C, T_\bullet)\rightarrow C\otimes_R F_\bullet$
is a quasiisomorphism.
As $C\otimes_R F_\bullet$ is a minimal finite complex, we are done.

{\rm (b).} Let $T_\bullet$ and $L_\bullet$ be two minimal finite
complexes of modules of the form $C^m$. Assume that
$\alpha:T_\bullet\rightarrow L_\bullet$ is a quasiisomorphism. Then,
by Remark~\ref{R10} and Lemma~\ref{L10}, $\mathrm{ Hom}_R(C,
\alpha): \mathrm{ Hom}_R(C, T_\bullet)\rightarrow \mathrm{ Hom}_R(C,
L_\bullet)$ is a quasiisomorphism of minimal finite complexes of
finite free $R$--modules. Thus, by the proof of \cite[Lemma
1.6(2)]{bf}, $\mathrm{ Hom}_R(C, \alpha)$ is an isomorphism. Now,
there is a commutative diagram of complexes and morphisms
$$\begin{array}{ccc}
 T_\bullet & \stackrel{\simeq}{\underset{\alpha}{-\hspace{-0.2cm}-\hspace{-0.2cm}-\hspace{-0.2cm}-\hspace{-0.2cm}-\hspace{-0.2cm}
 -\hspace{-0.2cm}-\hspace{-0.2cm}-\hspace{-0.2cm}-\hspace{-0.2cm}-\hspace{-0.2cm}-\hspace{-0.2cm}-\hspace{-0.2cm}
 \longrightarrow}} & L_\bullet  \\
  \begin{array}{ll} \ \ {\Big\uparrow} \cong \end{array} &     &\begin{array}{ll} \ \ {\Big\uparrow} \cong \end{array}  \\
 C\otimes_R \mathrm{ Hom}_R(C, T_\bullet)
 & \stackrel{\cong}{\underset{C\otimes_R \mathrm{ Hom}_R(C, \alpha)}{-\hspace{-0.2cm}-\hspace{-0.2cm}-\hspace{-0.2cm}
 -\hspace{-0.2cm}-\hspace{-0.2cm}-\hspace{-0.2cm}-\hspace{-0.2cm}-\hspace{-0.2cm}-\hspace{-0.2cm}\longrightarrow}} &
  C\otimes_R \mathrm{ Hom}_R(C, L_\bullet), \\
\end{array}$$
where the vertical morphisms are natural isomorphisms. This implies
that $\alpha$ itself must be an isomorphism.
\end{proof}

It is proved in \cite[Lemma 4.9]{w} that every finite module $M$ over a commutative noetherian local ring $R$ with
$\mathcal P^f_C$--$\mathrm{ pd}(M)<\infty$ admits a minimal $\mathcal P_C^f$--resolution.
Now we show that every finite $R$--module which has a proper $\mathcal P_C$--resolution, admits a minimal proper one.
Note that if $\mathcal P^f_C$--$\mathrm{ pd}(M)<\infty$ then $M$ admits a proper $\mathcal P_C$--resolution
(see proof of \cite[Corollary 2.10]{tw}).

\begin{thm}\label{E}
Assume that $R$ is a commutative noetherian local ring and that $C$ is a semidualizing $R$--module. Then $\mathcal P_C^f$ is covering in the category of finite $R$--modules. For any finite $R$--module $M$, there is a complex
$X=\cdots \longrightarrow C^{n_1}\longrightarrow C^{n_0}\longrightarrow 0$ with the following properties.

$\mathrm{ (1)}$ $X^+=\cdots \longrightarrow C^{n_1}\longrightarrow C^{n_0}\longrightarrow M \longrightarrow 0$ is
$\mathrm{ Hom}_R(\mathcal P_C, -)$--exact.

$\mathrm{ (2)}$ $X$ is a minimal complex.

If $M$ admits a proper $\mathcal P_C$--resolution, then $X^+$ is exact and so $X$ is a minimal proper $\mathcal P_C$--resolution of $M$.
\end{thm}
\begin{proof}
Let $M$ be a finite $R$--module. Assume that $n_0=\nu (\mathrm{ Hom}_R(C, M))$ denotes
the number of a minimal set of generators of $\mathrm{ Hom}_R(C, M)$
and that $\alpha:R^{n_0}\longrightarrow  \mathrm{ Hom}_R(C, M)$ is the natural epimorphism.
As $\alpha$ is a $\mathcal P^f$--cover of $\mathrm{ Hom}_R(C, M)$, the natural map
$\beta=C\otimes_RR^{n_0}\stackrel{C\otimes_R\alpha}{-\hspace{-0.07cm}-\hspace{-0.17cm}\longrightarrow}
C\otimes_R\mathrm{ Hom}_R(C, M) \stackrel{\nu_M}{\longrightarrow}M$
is a $\mathcal P_C^f$--cover of $M$.
Set $M_1=\mathrm{ Ker} \beta$ and $n_1=\nu(\mathrm{ Hom}_R(C, M_1))$. Thus there is a $\mathcal P_C^f$--cover
$\beta_1: C\otimes_RR^{n_1}\longrightarrow M_1$.
Proceeding in this way one obtains a complex
$$X=\cdots \stackrel{\partial_2=\epsilon_2 \beta_2}{-\hspace{-0.07cm}-\hspace{-0.17cm}\longrightarrow} C\otimes_RR^{n_1}\stackrel{\partial_1=\epsilon_1 \beta_1}{-\hspace{-0.07cm}-\hspace{-0.17cm}\longrightarrow} C\otimes_RR^{n_0}\longrightarrow 0,$$
where $\epsilon_i: M_i \rightarrow C\otimes_RR^{n_{i-1}}$ is the inclusion map for all $i\geqslant 1$.
As the maps in $X$ are obtained by $\mathcal P_C^f$--covers, the complex $X^+$ is $\mathrm{ Hom}_R(\mathcal P_C, -)$--exact.
It is easy to see that $\mathrm{ Hom}_R(C, X)$ is minimal free resolution of $\mathrm{ Hom}_R(C, M)$.
Now we show that $X$ is a minimal complex. Let $f: X\rightarrow X$ be a morphism which is homotopic to $\mathrm{id}_X$.
It is easy to see that the morphism $\mathrm{Hom}_R(C, f)$ is homotopic to $\mathrm{id}_{\mathrm{Hom}_R(C, X)}$.
As the complex $\mathrm{ Hom}_R(C, X)$ is minimal, by \cite[Proposition 1.7]{am}, the morphism $\mathrm{Hom}_R(C, f)$ is an isomorphism.
The commutative diagram
$$\begin{array}{ccc}
 X & \stackrel{f}{{-\hspace{-0.2cm}-\hspace{-0.2cm}-\hspace{-0.2cm}-\hspace{-0.2cm}-\hspace{-0.2cm}
 -\hspace{-0.2cm}-\hspace{-0.2cm}-\hspace{-0.2cm}-\hspace{-0.2cm}-\hspace{-0.2cm}-\hspace{-0.2cm}-\hspace{-0.2cm}
 \longrightarrow}} & X  \\
  \begin{array}{ll} \ \ {\Big\downarrow} \cong \end{array} &     &\begin{array}{ll} \ \ {\Big\downarrow} \cong \end{array}  \\
 C\otimes_R \mathrm{ Hom}_R(C, X)
 & \stackrel{\cong}{\underset{C\otimes_R \mathrm{ Hom}_R(C, f)}{-\hspace{-0.2cm}-\hspace{-0.2cm}-\hspace{-0.2cm}
 -\hspace{-0.2cm}-\hspace{-0.2cm}-\hspace{-0.2cm}-\hspace{-0.2cm}-\hspace{-0.2cm}-\hspace{-0.2cm}\longrightarrow}} &
  C\otimes_R \mathrm{ Hom}_R(C, X), \\
\end{array}$$
with vertical natural isomorphisms, implies that $f$ is an
isomorphism. Therefore, by \cite[Proposition 1.7]{am}, $X$ is
minimal. If $M$ admits a proper $\mathcal P_C$--resolution, then by
\cite[Corollary 2.3]{tw}, $X^+$ is exact.
\end{proof}
The proof of the next lemma is similar to \cite[Corollary 2.3]{tw}.

\begin{lem}\label{g}
Let $R$ be a commutative noetherian ring and let $M$ be a finite $R$--module. Assume that $C$ is a semidualizing $R$--module.
The following are equivalent.
\begin{itemize}
\item[(i)] $M$ admits a proper $\mathcal P^f_C$--coresolution.
\item[(ii)]  Every $\mathrm{ Hom}_R(-, \mathcal P^f_C)$--exact complex of the form
$$0\longrightarrow M \longrightarrow C\otimes_RQ_0 \longrightarrow C\otimes_RQ_{-1}\longrightarrow \cdots$$
is exact, where $Q_i$ is an object of $\mathcal P^f$ for all $i\leqslant0$.
\item[(iii)] The natural homomorphism $M \longrightarrow \mathrm{ Hom}_R(\mathrm{ Hom}_R(M, C), C)$ is an isomorphism and $\mathrm{ Ext}^{\geqslant1}_R(\mathrm{ Hom}_R(M, C), C)=0$.
\end{itemize}
\end{lem}

\begin{prop}\label{F}
Assume that $R$ is a commutative noetherian local ring and that $C$ is a semidualizing $R$--module. Then $\mathcal P_C^f$ is enveloping in the category of finite $R$--modules. For any finite $R$--module $M$, there is a complex
$Y=0 \longrightarrow C^{m_0}\longrightarrow C^{m_1}\longrightarrow \cdots$ with the following properties.

$\mathrm{ (1)}$ $ ^+Y=0\longrightarrow M\longrightarrow  C^{m_0}\longrightarrow  C^{m_1} \longrightarrow \cdots $ is $\mathrm{ Hom}_R(-, \mathcal P_C)$--exact.

$\mathrm{ (2)}$ $Y$ is a minimal complex.

If $M$ admits a proper $\mathcal P^f_C$--coresolution, then $^+Y$ is exact and so $Y$ is a minimal proper $\mathcal P_C$--coresolution of $M$.
\end{prop}
\begin{proof}
Let $M$ be a finite $R$--module. Assume that $m_0=\nu (\mathrm{
Hom}_R(M, C))$ denotes the number of a minimal set of generators of
$\mathrm{ Hom}_R(M, C)$ and that $\alpha:R^{m_0}\longrightarrow
\mathrm{ Hom}_R(M, C)$ is the natural $\mathcal P^f$--cover of
$\mathrm{ Hom}_R(M, C)$. It follows that
$\gamma=M\stackrel{\delta_M}{\longrightarrow}\mathrm{
Hom}_R(\mathrm{ Hom}_R(M, C), C)\stackrel{\mathrm{ Hom}_R(\alpha,
C)}{-\hspace{-0.07cm}-\hspace{-0.17cm}\longrightarrow}\mathrm{
Hom}_R(R^{m_0}, C)$ is a $\mathcal P_C^f$--envelope of $M$. Set
$M_{-1}=\mathrm{ Coker} \gamma$ and $m_1=\nu (\mathrm{
Hom}_R(M_{-1}, C))$. As mentioned, there is a $\mathcal
P_C^f$--envelope $\gamma_1:M_{-1}\longrightarrow\mathrm{
Hom}_R(R^{m_1}, C)$. Proceeding in this way one obtains a complex
$Y=0 \longrightarrow \mathrm{ Hom}_R(R^{m_0},
C)\stackrel{\partial_0=\gamma_1\pi_1}{-\hspace{-0.07cm}-\hspace{-0.17cm}\longrightarrow}
\mathrm{ Hom}_R(R^{m_1},
C)\stackrel{\partial_{-1}=\gamma_2\pi_2}{-\hspace{-0.07cm}-\hspace{-0.17cm}\longrightarrow}
\cdots$, where $\pi_i$ is the natural epimorphism for all
$i\geqslant 1$. Since the maps in $Y$ are obtained by $\mathcal
P_C^f$--envelopes, the complex $ ^+Y$ is  $\mathrm{ Hom}_R(-,
\mathcal P_C)$--exact. It is easy to see that $\mathrm{ Hom}_R(Y, C)$
is minimal free resolution of $\mathrm{ Hom}_R(M, C)$.
Similar to the proof of Theorem~\ref{E}, we find that $Y$ is a minimal complex.
If $M$ admits a proper $\mathcal P^f_C$--coresolution,
then, by Lemma \ref{g}, $^+Y$ is exact.
\end{proof}

In the following example we find an $R$--module $M$ with $\mathcal
P_C$--$\mathrm{ pd}(M)=\infty$ which admits a minimal proper
$\mathcal P_C$--resolution. This example shows that a commutative
noetherian local ring which admits an exact zero-divisor is not a
strongly regular ring.

\begin{exmp}\label{ex3}
Let $R$ be a commutative noetherian local ring, $C$ a semidualizing $R$--module. 
Assume that $x,  y$ form a pair of exact zero-divisors on both $R$ and $C$ (e.g. see \cite[Example 3.2]{ad}).
Then $\mathcal P_C$--$\mathrm{ pd}(C/xC)=\mathrm{ pd}(R/xR)=\infty$.  The complex
$$T_\bullet=\cdots \stackrel{x}{\longrightarrow}C \stackrel{y}{\longrightarrow}C\stackrel{x}{\longrightarrow}C\longrightarrow 0\ (\ \mathrm{ resp.}\ L_\bullet= 0\longrightarrow C \stackrel{x}{\longrightarrow}C \stackrel{y}{\longrightarrow}C\stackrel{x}{\longrightarrow}\cdots )$$
is a minimal $\mathcal P_C$--resolution (resp. $\mathcal P_C$--coresolution) of $C/{xC}$.
By \cite[Proposition 3.4]{ad}, $C/xC$ is a semidualizing $R/xR$--module. By \cite[Proposition 2.13]{dg}, there are isomorphisms
$$\mathrm{ Hom}_R(C, C/{xC})\cong \mathrm{ Hom}_{R/xR}(C/{xC}, C/{xC})\cong R/xR,$$
$$\mathrm{ Hom}_R(C/{xC}, C)\cong \mathrm{ Hom}_{R/xR}(C/{xC}, C/{xC})\cong R/xR.$$
Applying $\mathrm{ Hom}_R(C, -)$ and $\mathrm{ Hom}_R(-, C)$ on the above complexes, respectively, would result the isomorphisms $\mathrm{ Hom}_R(C, T_\bullet^+)\cong F_\bullet^+$ and $\mathrm{ Hom}_R( ^+L_\bullet, C)\cong F_\bullet^+$, where
$F_\bullet^+$ is the exact complex
$\cdots \stackrel{y}{\longrightarrow}R\stackrel{x}{\longrightarrow}R \stackrel{y}{\longrightarrow}R\stackrel{x}{\longrightarrow}R\longrightarrow R/xR\longrightarrow 0.$
Therefore $T_\bullet$ (resp. $L_\bullet$) is a minimal proper $\mathcal P_C$--resolution (resp. $\mathcal P_C$--coresolution) of $C/{xC}$.

For each $n$, one obtains a $C$--perfect complex of length $n$ as
$$T_\bullet^{(n)}= 0 \longrightarrow C {\longrightarrow}C {\longrightarrow}\cdots\stackrel{x}{\longrightarrow}C\stackrel{y}{\longrightarrow}C\stackrel{x}{\longrightarrow}C \longrightarrow 0,$$
where $T_i^{(n)}= T_i$ for all $0\leq i \leq n$ and $T_i^{(n)}=0$ otherwise.
Note that the induced map $\bar{d}_i:T_i^{(n)}/{\mathrm{ Ker}\, d_i}\rightarrow T_{i-1}^{(n)}$ is injective,
where $\mathrm{ Ker}\, d_i$ is equal to $yC$ or $xC$. As $C$ is indecomposable $R$--module, $T_\bullet^{(n)}$ is indecomposable which has a similar proof to \cite[Proposition 1.5]{bf}.
\end{exmp}

\section*{\bf Acknowledgment}
The authors are grateful to the referee for his/her careful reading
of the paper and valuable comments. The second author was supported
in part by a grant from IPM (No.93130110).




\begin{thebibliography}{10}

\bibitem {ad} E. Amanzadeh and M. T. Dibaei,  Auslander class, $\mbox{G}_C$ and $C$--projective modules modulo exact zero-divisors,
{\it Comm. Algebra}, to appear.


\bibitem {am} L. L. Avramov and A. Martsinkovsky,  Absolute, relative, and Tate cohomology of modules of finite Gorenstein dimension,
{\it Proc. London Math. Soc.}  {\bf 85} (2002), no.3, 393--440.


\bibitem {bf} R-O. Buchweitz and H. Flenner,  Strong global dimension of commutative rings and schemes, {\it J. Algebra} {\bf 422} (2015), 741--751.


\bibitem {cf} L. W. Christensen and H. B. Foxby, Hyperhomological algebra with applications to commutative rings, http://www.math.ttu.edu/\~lchriste/download/918-final.pdf

\bibitem{dg} M. T. Dibaei and M. Gheibi,   Sequence of exact zero--divizors,  arXiv:1112.2353v3 (2012).


\bibitem {ej} E. E. Enochs and O. M. G. Jenda,  Relative homological algebra, Walter de Gruyter. Berlin. New York 2000.


\bibitem {gm} S. I. Gelfand and Y. I. Manin,  Methods of homological algebra, Springer Monographs in Mathematics, 1988.


\bibitem {hj} H. Holm and P. J{\o}rgensen,  Semi-dualizing modules and related Gorenstein homological dimensions, {\it J. Pure Appl. Algebra}
{\bf 205} (2006), 423--445.

\bibitem {hw}  H. Holm and D. White,   Foxby equvalence over associative rings, {\it J. Math. Kyoto Univ.} {\bf 47} (2007), no. 4, 781--808.


\bibitem {r} J. J. Rotman,   An introduction to homological algebra,  Springer Universitext, Second Edition, 2009.


\bibitem{s-w} S. Sather-Wagstaff,   Semidualizing modules, http://www.ndsu.edu/pubweb/\~ssatherw/DOCS/sdm.pdf

\bibitem {wsw1} S. Sather-Wagstaff, T. Sharif and D. White,  Stability of Gorenstein categories, {\it J. Lond. Math. Soc.}
{\bf 77} (2008), no. 2, 481--502.

\bibitem {tw} R. Takahashi and D. White,   Homological aspects of semidualizing modules, {\it Math. Scand.} {\bf 106} (2010), 5--22.


\bibitem {w} D. White,  Gorenstein projective dimension with respect to a semidualizing module, {\it J. Commut. Algebra}
{\bf 2} (2010), no. 1, 111--137.

\end{thebibliography}
\end{document}